\documentclass[preprint,12pt]{elsarticle}

\usepackage[utf8]{inputenc}
\usepackage{amssymb}
\usepackage{amsmath}
\usepackage{amsthm}
\usepackage{arydshln}
\usepackage{titletoc}
\usepackage{titlesec}
\usepackage{subfigure}
\usepackage{epsfig}

\usepackage{graphicx,amsmath}
\usepackage{algpseudocode}
\usepackage{algorithmicx,algorithm}
\usepackage{colortbl}
\usepackage{tabularx}
\usepackage{booktabs}
\usepackage{threeparttable}
\usepackage{arydshln}
\usepackage{multirow}
\usepackage{pgfplots}

\begin{document}

\begin{frontmatter}



\newtheorem{proposition}{Proposition}[section]
\newtheorem{theorem}{Theorem}[section]
\newtheorem{definition}{Definition}[section]
\newtheorem{lemma}{Lemma}[section]
\newtheorem{remark}{Remark}[section]
\newtheorem{assumption}{Assumption}[section]
\newtheorem{example}{Example}[section]
\newtheorem{corollary}{Corollary}[section]

\title{Last-iterate convergence analysis of stochastic momentum methods for
neural networks\tnoteref{t1}}

\tnotetext[t1]{This work was funded in part by the National Natural Science Foundation of
China (No. 62176051), in part by National Key R\&D Program of China (No. 2020YFA0714102), and in part by the Fundamental Research Funds for the
Central Universities of China. (No. 2412020FZ024).}
\author[nenu]{Dongpo Xu\fnref{label2}}
\ead{xudp100@nenu.edu.cn}
\fntext[label2]{The authors contributed equally to this work.}
\author[nenu]{Jinlan Liu\fnref{label2}}
\author[chsc]{Yinghua Lu\corref{cor1}}
\ead{luyh@nenu.edu.cn}
\author[kasnenu]{Jun Kong\corref{cor1}}
\ead{kongjun@nenu.edu.cn}
\author[danilo]{Danilo P. Mandic}
\ead{d.mandic@imperial.ac.uk}

\cortext[cor1]{Corresponding authors}
\address[nenu]{School of Mathematics and Statistics, Northeast Normal University, Changchun 130024, China}
\address[chsc]{Institute for Intelligent Elderly Care, Changchun Humanities and Sciences College,
Changchun 130117, China}
\address[kasnenu]{Key Laboratory of Applied Statistics of MOE, Northeast Normal University, Changchun 130024, China}
\address[danilo]{Department of Electrical and Electronic Engineering, Imperial College London, SW7 2AZ London, UK}

\begin{abstract}
The stochastic momentum method is a commonly used acceleration technique for solving large-scale stochastic optimization problems in artificial neural networks. Current convergence results of stochastic momentum methods under non-convex stochastic settings mostly discuss convergence in terms of the random output and minimum output. To this end, we address the convergence of
the last iterate output (called \textit{last-iterate convergence}) of the stochastic momentum methods for non-convex stochastic optimization problems,
in a way conformal with traditional optimization
theory. We prove the last-iterate convergence of the stochastic momentum methods under a unified framework, covering
both stochastic heavy ball momentum and stochastic Nesterov accelerated gradient momentum. The momentum factors can be fixed to be constant, rather than time-varying coefficients in existing analyses. Finally, the last-iterate convergence of the stochastic momentum methods is verified on the benchmark MNIST and CIFAR-10 datasets.
\end{abstract}


\end{frontmatter}


\section{Introduction.}

We consider non-convex  stochastic optimization problems of the form
\begin{equation}\label{eq:optfprobmcons}
\min_{x\in\mathbb{R}^d}f(x):=\mathbb{E}_{\xi}[\ell(x,\xi)],
\end{equation}
where $\ell$ is a smooth non-convex function and $\mathbb{E}_{\xi}[\cdot]$ denotes the statistical expectation with respect to the random variable $\xi$.
Optimization problems of this form arise naturally in machine learning where $x\in\mathbb{R}^d$ are the parameters of neural networks \cite{LeCun,Schmidhuber,Mandicbk}, $\ell$ represents a loss from individual training examples
or mini-batches and $f$ is the full training objective
function. One of the most popular algorithms for solving such optimization problem is stochastic gradient descent (SGD) \cite{Robbins,Luo22}. The
advantages of SGD for large scale stochastic optimization and the related issues of tradeoffs between computational
and statistical efficiency was highlighted in \cite{Bottou07}. Starting from $x_1\in\mathbb{R}^d$, SGD updates the parameters of the network model until convergence
\begin{equation}
x_{t+1}=x_{t}-\eta_tg_t,
\end{equation}
where $\eta_t>0$ is the stepsize and $g_t$ is a noisy gradient such that $\mathbb{E}_t[g_t]=\nabla f(x_t)$ is an unbiased estimator of the full gradient $\nabla f(x_t)$. Classical convergence analysis of the SGD has established that in order for $\lim_{t\to\infty}\|\nabla f(x_t)\|=0$, the stepsize evolution should satisfy
\begin{equation}\label{eq:etacondi}
\sum_{t=1}^\infty\eta_t=\infty,\quad\quad \sum_{t=1}^\infty\eta_t^2<\infty\;.
\end{equation}
However, the basic SGD (also called vanilla SGD) suffers from slow convergence, and to this end a variety of techniques have been introduced to improve convergence
speed, including adaptive stepsize methods \cite{Duchi,Tieleman, Xu21}, momentum  acceleration methods \cite{Nesterov,Polyak,GhadimiE} and adaptive gradient methods \cite{Kingma, Luo, Xu22}.  Among these algorithms, momentum methods are particularly desirable, since they merely require slightly more computation per
iteration. Heavy Ball (HB) \cite{Polyak,GhadimiE} and Nesterov
Accelerated Gradient (NAG) \cite{Nesterov} are two most popular momentum methods.
In non-convex stochastic optimization setting, existing convergence results of the momentum methods usually have the following form \cite{Ghadimi, FZou,YYan}
\begin{equation}
\lim_{T\to\infty}\min_{t\in[T]}\mathbb{E}\|\nabla f(x_t)\|^2=0,
\end{equation}
or
\begin{equation}
\lim_{T\to\infty}\mathbb{E}\|\nabla f(x_\tau)\|^2=0,
\end{equation}
where $x_\tau$ is an iterate randomly chosen from $\{x_1, x_2$, $\cdots,x_T\}$ under some probability distribution.
The most common type is the uniform distribution, where we have $\lim_{T\to\infty}\frac{1}{T}\sum_{t=1}^T\mathbb{E}\|\nabla f(x_t)\|^2=0$.
Note that these two convergence properties are weaker than the usual convergence, given by
\begin{equation}
\lim_{t\to\infty}\mathbb{E}\|\nabla f(x_t)\|^2=0,
\end{equation}
where the convergence of the last iterate of the momentum methods is referred
to as the \textit{last-iterate convergence}.

In this work, we revisit the momentum methods with the aim to answer two basic questions:
\begin{itemize}
\item[1)] Can the stochastic momentum methods archive the last-iterate convergence in the non-convex setting?
\item[2)] Can the momentum coefficient in the convergence analysis be fixed as a constant commonly used in practice?
\end{itemize}

Our analysis provides affirmative answers to both questions, with the main contributions of this work summarized as follows:
\begin{itemize}
\item[$\bullet$] Last-iterate convergence of the stochastic momentum methods is proven for the \textit{first} time in the non-convex setting, without bounded weight assumption. This theoretically supports the intuition of usually choosing the last iterate, rather than the minimum or random selection that relies on storage during the iterations \cite{FZou, YYan, ChenX, Zou}.
\item[$\bullet$] The convergence condition for the momentum coefficient is extended to be a fixed constant, rather than time-varying coefficients (e.g., $\sum_{t=1}^\infty\mu_t^2<\infty$ or even more complicated) with bounded weight assumption \cite{Wangj,Zhangnm}, which mirrors the actual implementations of the stochastic momentum method in deep learning libraries \cite{Stevens,Zaccone}.
\item[$\bullet$] Numerical experiments support the theoretical findings of the last-iterate convergence of the stochastic momentum methods,
and show that the stochastic momentum methods have good convergence performance on the benchmark datasets, while exhibiting good robustness for different interpolation factors.
\end{itemize}

The rest of this paper is organized as follows.
Section 2 introduces a unified view of the stochastic momentum methods.
The main convergence results with the rigorous proofs are presented in
Section 3. Experimental results are reported in Section 4. Finally, we conclude this work.\\

\noindent{\bf Notations.} We use $\mathbb{E}[\cdot]$ to denote the statistical expectation,
and $\mathbb{E}_t[\cdot]$ as the conditional expectation with respect to $g_{t}$ conditioned on the the previous $g_1,g_2,\cdots,g_{t-1}$,
$[T]$ denotes the set $\{1,2,\cdots,T\}$, while norm
$\|x\|$ is designated by $\|x\|_2$ if not otherwise specified.

\section{Stochastic unified momentum methods}
We next study the stochastic HB  and NAG using a unified formulation. The stochastic HB (SHB) is characterized by the iteration
\begin{equation}
{\rm SHB}: x_{t+1}=x_{t}-\eta_tg_{t}+\mu(x_{t}-x_{{t-1}}),
\end{equation}
with $x_0=x_1\in\mathbb{R}^d$, where $\mu$ is the momentum constant and $\eta_t$ is the step size.
By introducing $m_t=x_{t+1}-x_{t}$ with $m_0=0$, the above update becomes
\begin{equation}
{\rm SHB}:\quad
    \begin{cases}
      m_{t}=\mu m_{t-1}-\eta_tg_{t}\\
      x_{t+1}=x_{t}+ m_{t}
    \end{cases}.
\end{equation}
The update of stochastic NAG (SNAG) is given by
\begin{equation}
{\rm SNAG}:\quad
    \begin{cases}
      y_{t+1}=x_{t}-\eta_tg_{t}\\
      x_{t+1}=y_{t+1}+\mu (y_{t+1}-y_t)
    \end{cases},
\end{equation}
with $x_1=y_1\in\mathbb{R}^d$. By introducing  $m_t=y_{t+1}-y_{t}$ with $m_0=0$, the iterate of SNAG can be equivalently written as
\begin{equation}
{\rm SNAG}:\quad
    \begin{cases}
      m_{t}=\mu m_{t-1}-\eta_tg_{t}\\
      x_{t+1}=x_{t}-\eta_tg_{t}+\mu m_t
    \end{cases}.
\end{equation}
Then, SHB and SNAG can be rewriten in the following stochastic unified momentum (SUM) form
\begin{equation}\label{eq:mupdaterl}
{\rm SUM}:\quad
    \begin{cases}
      m_{t}=\mu m_{t-1}-\eta_tg_{t}\\
      x_{t+1}=x_{t}-\lambda\eta_t g_{t}+(1-\tilde{\lambda}) m_{t}
    \end{cases},
\end{equation}
where $\mu\in[0,1)$ is the momentum constant, $\lambda\geq0$ is a interpolation factor
and $\tilde{\lambda}:=(1-\mu)\lambda$. When $\lambda=0$, we have SHB; when $\lambda=1$, we have SNAG.
To ensure that $1-\tilde{\lambda}$ is non-negative, $\lambda$ should be selected from the interval $[0,1/(1-\mu)]$.

\begin{remark}
Zou et al. \cite{FZou} unified SHB and SNAG as a two-step iterative scheme, given by
\begin{equation}\label{eq:zoumupdaterl}
m_{t}=\mu m_{t-1}-\eta_tg_{t},\quad x_{t+1}=x_{t}+m_t+\lambda\mu(m_t-m_{t-1}),
\end{equation}
with $m_0=0$. Note that \eqref{eq:mupdaterl} and \eqref{eq:zoumupdaterl} are completely equivalent in a mathematical sense. However, there is a subtle difference in the actual implementation. The iteration in \eqref{eq:mupdaterl} only needs the knowledge of $m_t$ at time $t$, while the iteration in \eqref{eq:zoumupdaterl} requires both $m_{t-1}$ and $m_t$ at the same time. In addition, the convergence of the iteration in \eqref{eq:zoumupdaterl} is characterized by $$\lim_{T\to\infty}\frac{1}{T}\sum_{t=1}^T\left(\mathbb{E}\|\nabla f(x_t)\|^{4/3}\right)^{3/2}=0,$$ which is weaker than the last-iterate convergence $\lim_{t\to\infty}\mathbb{E}\|\nabla f(x_t)\|^2=0$.
\end{remark}

\begin{algorithm}
\caption{SUM with the last-iterate output}\label{alg:sum1}
{\bf Input:} Initial point $x_1\in \mathbb{R}^d$, step size $\{\eta_t\}_{t\geq1}$ and $\mu\in[0,1)$
\begin{algorithmic}[1]
\State Set $m_0=0$ and $\lambda\in[0,1/(1-\mu)]$
\For{$t=1,2,...,T$}
    \State Get an unbiased noisy gradient $g_t$
    \State $m_{t}=\mu m_{t-1}-\eta_tg_{t}$
    \State $x_{t+1}=x_{t}-\lambda\eta_t g_{t}+(1-\lambda+\lambda\mu) m_{t}$
\EndFor
\end{algorithmic}
{\bf Output:} The last iterate $x_T$.
\end{algorithm}

\begin{remark}
Yan et al. \cite{YYan} unified SHB and SNAG as a three-step iterative scheme as follows
\begin{equation}\label{eq:yanmupdaterl}
\begin{split}
&y_{t+1}=x_t-\eta_tg_t,\quad \tilde{y}_{t+1}=x_t-\lambda\eta_tg_t,\\
&x_{t+1}=y_{t+1}+\mu(\tilde{y}_{t+1}-\tilde{y}_{t}),
\end{split}
\end{equation}
with $\tilde{y}_1=x_1$. Note that \eqref{eq:mupdaterl} is slightly more economical than \eqref{eq:yanmupdaterl}, in terms of computation and storage.
In addition, the convergence form of the iteration in \eqref{eq:yanmupdaterl} is $\lim_{T\to\infty}\min_{t\in[T]}\mathbb{E}\|\nabla f(x_t)\|^2=0$, which is
 a simple corollary of last-iterate convergence $\lim_{t\to\infty}\mathbb{E}\|\nabla f(x_t)\|^2=0$.
\end{remark}

Next, we introduce three different forms of the SUM method in terms of their convergence behaviour. Algorithm 1 corresponds to last-iterate convergence, Algorithm 2 corresponds to random selection convergence \cite{FZou}, while Algorithm 3 corresponds to the minimum gradient convergence \cite{YYan}.

\begin{algorithm}
\caption{SUM with random output}
{\bf Input:} Initial point $x_1 \in \mathbb{R}^d$, step size $\{\eta_t\}_{t\geq1}$ and $\mu\in[0,1)$.
\begin{algorithmic}[1]
\State Set $\mathcal{S}=\emptyset$, $m_0=0$ and $\lambda\in[0,1/(1-\mu)]$
\For{$t=1,2,...,T$}
    \State $\mathcal{S}=\mathcal{S}\cup\{x_{t}\}$
    \State Get an unbiased noisy gradient $g_t$
    \State $m_{t}=\mu m_{t-1}-\eta_tg_{t}$
    \State $x_{t+1}=x_{t}+m_t+\lambda\mu (m_{t}-m_{t-1})$
\EndFor
\end{algorithmic}
{\bf Output:} Choose $x_\tau$ from the set $\mathcal{S}$ with some probability distribution.
\end{algorithm}

\begin{algorithm}[t]
\caption{SUM with minimum output}
{\bf Input:} Initial point $x_1 \in \mathbb{R}^d$, step size $\{\eta_t\}_{t\geq1}$ and $\mu\in[0,1)$.
\begin{algorithmic}[1]
\State Set $x_{\tau}=\tilde{y}_1=x_1$, $v=\infty$ and $\lambda\in[0,1/(1-\mu)]$
\For{$t=1,2,...,T$}
    \State Calculate the full gradient $\nabla f(x_{t})$
	\If{$\|\nabla f(x_{t})\|<v$}
        \State $x_{\tau} =x_{t}$, $v=\|\nabla f(x_{t})\|$
	\EndIf
    \State Get an unbiased noisy gradient $g_t$
    \State $y_{t+1}=x_t-\eta_tg_t$
    \State $\tilde{y}_{t+1}=x_t-\lambda\eta_t g_t$
    \State $x_{t+1}=y_{t+1}+\mu(\tilde{y}_{t+1}-\tilde{y}_{t})$
\EndFor
\end{algorithmic}
{\bf Output:}  The minimum iterate $x_{\tau}$ that satisfies $\tau=\arg\min_{t\in[T]} \|\nabla f(x_{t})\|$.
\end{algorithm}

\begin{remark}
It is worth mentioning that the three SUM methods are mathematically equivalent,
but they may be slightly different in the actual implementation.
The main difference lies in the output mode.
Algorithm 1 and Algorithm 2 are basically the same regarding computational efficiency, but all updated parameters of Algorithm 2 need to be stored in memory. In addition, it is difficult to obtain the optimal parameters through the random selection of parameters in the actual implementation. Algorithm 3 makes it easy to archive the optimal parameters, but the computational cost of the full gradient in each iteration is too high to be adopted in practice.  If the last-iterate convergence of Algorithm 1 can be guaranteed, then Algorithm 1 with the last iterate is undoubtedly the best tradeoff scheme, because it has the advantages in both computation and storage.
\end{remark}

\section{Main results}

\subsection{Technical lemmas}
We now provide some lemmas, which are pivotal in the proof of last-iterate convergence of the proposed SUM method.
\begin{lemma}\label{lem:abcabp}
Let $\{ a_n\}_{n\geq1}$, $\{ b_n\}_{n\geq1}$ and $\{\tilde{a}_n\}_{n\geq1}$ be three non-negative real sequences such that $\sum_{n=1}^\infty a_n=\infty$,
$\sum_{n=1}^\infty a_nb_n^p<\infty$, $\lim_{n\to\infty}{a_n}/{\tilde{a}_n}=1$, and $|b_{n+1}-b_n|\leq C \tilde{a}_nb_n^{p-\epsilon}$ for some positive constants $C$, $p$ and $\epsilon\in[0,p]$.  Then, we have
\begin{equation}
\lim_{n\to\infty} b_n=0\,.
\end{equation}
\begin{proof}
Since $\sum_{n=1}^\infty a_nb_n^p$ converges, we necessarily have $\lim \inf_{n\to\infty} b_n=0$. Otherwise, it would contradict the
assumption $\sum_{n=1}^\infty a_n$ diverges.

We now proceed to establish the proof by contradiction, and assume that $\lim \sup_{n\to\infty} b_n\geq \nu >0$.
Let $\{m_k\}_{k\geq1}$, $\{n_k\}_{k\geq1}$ be sequence of indexes such that
$$ m_k<n_k<m_{k+1}$$,
\begin{equation}\label{eq:3epsilon}
\frac{\nu}{2}<b_n,\quad \text{ for $m_k\leq n<n_k$}\,,
\end{equation}
\begin{equation}\label{eq:3epsilonleq}
b_n\leq\frac{\nu}{2},\quad \text{ for $n_k\leq n<m_{k+1}$}\,.
\end{equation}
Since $\sum_{n=1}^\infty a_nb_n^p$ converges and $\lim_{n\to\infty}{a_n}/{\tilde{a}_n}=1$, there exists sufficiently large $\tilde{k}$ such that
\begin{equation}\label{eq:sertail}
\sum_{j=m_{\tilde{k}}}^\infty \tilde{a}_jb_j^p<\frac{1}{2C}\left(\frac{\nu}{2}\right)^{\epsilon+1}\,.
\end{equation}
Then for all $k\geq \tilde{k}$ and all $n$ with $m_k\leq n\leq n_k-1$, we have
\begin{equation}\label{eq:bnkbneps3}
\begin{split}
|b_{n_k}-b_n|&\leq\sum_{j=n}^{n_k-1}|b_{j+1}-b_j|\leq C\sum_{j=n}^{n_k-1}\tilde{a}_jb_j^{p-\epsilon}\\
&\leq C\left(\frac{\nu}{2}\right)^{-\epsilon}\sum_{j=n}^{n_k-1}\tilde{a}_jb_j^p\\
&\leq C\left(\frac{\nu}{2}\right)^{-\epsilon}\frac{1}{2C}\left(\frac{\nu}{2}\right)^{\epsilon+1}=\frac{\nu}{4}\,,
\end{split}
\end{equation}
where the last two inequalities follow from \eqref{eq:3epsilon} and \eqref{eq:sertail}. Thus,
\begin{equation}
b_n\leq b_{n_k}+\frac{\nu}{4}\leq\frac{3\nu}{4},\quad  \forall k\geq \tilde{k}, m_k\leq n\leq n_k-1\,.
\end{equation}
Upon combining the previous inequality with \eqref{eq:3epsilonleq}, we have $b_n\leq 3\nu/4,\quad  \forall n\geq m_{\tilde{k}}$.
This contradicts $\lim \sup_{n\to\infty} b_n\geq \nu >0$. Therefore, $\lim_{n\to\infty} b_n=0$.
\end{proof}
\end{lemma}

From Lemma \ref{lem:abcabp}, we obtain the following corollary.
\begin{corollary}\label{cor:abcabp}
Let $\{ a_n\}_{n\geq1}$, $\{ b_n\}_{n\geq1}$ and $\{\tilde{a}_n\}_{n\geq1}$ be three non-negative real sequences such that $\sum_{n=1}^\infty a_n=\infty$,
$\sum_{n=1}^\infty a_nb_n^2<\infty$, $\lim_{n\to\infty}{a_n}/{\tilde{a}_n}=1$, and $|b_{n+1}-b_n|\leq C \tilde{a}_n$ for a positive constant $C$.  Then, we have
\begin{equation}
\lim_{n\to\infty} b_n=0\,.
\end{equation}
\end{corollary}

The following lemma follows from Lemma 1 in \cite{Bertsekas}, and greatly simplifies the proof of the last-iterate convergence of the proposed SUM method.
\begin{lemma}[\cite{Bertsekas}]\label{lem:xyz}
Let $Y_t$, $W_t$, and $Z_t$ be three sequences such that $W_t$ is nonnegative for all $t$. Also, assume that
\begin{equation}
Y_{t+1}\leq Y_t-W_t+Z_t,\quad t=0, 1, \cdots,
\end{equation}
and that the series $\sum_{t=0}^TZ_t$ converges as $T\to\infty$. Then, either $Y_t\to-\infty$ or else $Y_t$ converges to a finite value and $\sum_{t=0}^\infty W_t<\infty$.
\end{lemma}

\subsection{Last-iterate convergence of the SUM methods}

The following lemma is used to determine the convergence of the series of parameter increments.
\begin{lemma}\label{lem:smbound}
Assume that $\mathbb{E}\| g_t \|^2\leq  G^2$. If the stepsize sequence satisfies \eqref{eq:etacondi}, then we have
\begin{equation}
\sum_{t=1}^\infty\mathbb{E}\|x_{t+1}-x_t\|^2<\infty,\quad
\sum_{t=1}^\infty \left(\sum_{k=1}^{t-1}\mu ^{t-k}\mathbb{E}\|m_{k}\|^2\right)<\infty\,.
\end{equation}
\end{lemma}
\begin{proof}
From \eqref{eq:mupdaterl}, we have
\begin{equation}\label{eq:mtisumbag}
m_{t}=-\sum_{k=1}^t\mu ^{t-k}\eta_kg_{k}=-\sum_{k=0}^{t-1}\mu ^k\eta_{t-k}g_{t-k}\,,
\end{equation}
and
\begin{equation}\label{eq:mtisumbagst}
\begin{split}
\|x_{t+1}-x_t\|^2&\leq\|-\lambda\eta_t g_{t}+(1-\tilde{\lambda}) m_{t}\|^2\\
&\leq2\lambda^2\eta_t^2 \|g_{t}\|^2+2(1-\tilde{\lambda})^2\| m_{t}\|^2\,.
\end{split}
\end{equation}
By Jensen's inequality \cite{Linz}, we get
\begin{equation}
\begin{split}
\|m_{t}\|^2&\leq\left(\sum_{k=1}^t\mu ^{t-k}\eta_k\|g_{k}\|\right)^2\leq\left(\sum_{k=1}^t\mu ^{t-k}\right)\sum_{k=1}^t\mu ^{t-k}\eta_k^2\|g_{k}\|^2\,.
\end{split}
\end{equation}
Upon taking the total expectation, we have
\begin{equation}\label{eq:emt2g2}
\begin{split}
\mathbb{E}[\|m_{t}\|^2]&\leq\left(\sum_{k=1}^t\mu ^{t-k}\right)\sum_{k=1}^t\mu ^{t-k}\eta_k^2\mathbb{E}[\|g_{k}\|^2]\leq \frac{G^2}{(1-\mu )}\sum_{k=1}^t\mu ^{t-k}\eta_k^2\,,
\end{split}
\end{equation}
and
\begin{equation}
\begin{split}
\mathbb{E}[\|x_{t+1}-x_t\|^2]&\leq2\lambda^2\eta_t^2 \mathbb{E}[\|g_{t}\|^2]+2(1-\tilde{\lambda})^2\mathbb{E}[\| m_{t}\|^2]\\
&\leq2\lambda^2 G^2\eta_t^2 +2(1-\tilde{\lambda})^2\mathbb{E}[\| m_{t}\|^2]\,.
\end{split}
\end{equation}
From the expression for the stepsizes \eqref{eq:etacondi}, we arrive at
\begin{equation}
\sum_{t=1}^{\infty}\eta_t^2<\infty, \quad\quad \sum_{t=1}^{\infty}\mu ^{t-1}<\infty,\quad 0<\mu <1\,.
\end{equation}
Since that $\sum_{k=1}^{t}\eta_k^2\mu ^{t-k}$ is the Cauchy product of the previous two positive convergent series \cite{Rudin}, this yields
\begin{equation}\label{eq:sumsumetauk}
\sum_{t=1}^\infty \left(\sum_{k=1}^{t-1}\eta_k^2\mu ^{t-k}\right)<\infty\,.
\end{equation}
Upon combing \eqref{eq:emt2g2}-\eqref{eq:sumsumetauk}, we finally have
\begin{equation}
\sum_{t=1}^\infty\mathbb{E}[\|m_{t}\|^2]<\infty,\quad \sum_{t=1}^\infty\mathbb{E}[\|x_{t+1}-x_t\|^2]<\infty,
\end{equation}
and in a similar manner
\begin{equation}
\sum_{t=1}^\infty \left(\sum_{k=1}^{t-1}\mu ^{t-k}\mathbb{E}[\|m_{k}\|^2]\right)<\infty\,.
\end{equation}
This completes the proof.
\end{proof}

The following lemma is used to determine the consistency between the direction of momentum $m_t$ and the direction of the full negative gradient.
\begin{lemma}\label{lem:fxsinequ}
Assume that $\nabla f$ is $L$-Lipschitz continuous and $\mathbb{E}\| g_t \|^2\leq  G^2$. Then for $x_t$ generated by SUM (Algorithm \ref{alg:sum1}), we have the following bound
\begin{equation}
\begin{split}
\mathbb{E}\left[\nabla f(x_t)^Tm_{t}\right]&\leq -\sum_{k=1}^{t}\mu ^{t-k}\eta_k\mathbb{E}[\|\nabla f(x_k)\|^2]\\
&\quad +2L\sum_{k=1}^{t-1}\mu ^{t-k}\mathbb{E}\left\|m_{k}\right\|^2
+L\lambda^2G^2\sum_{k=1}^{t-1}\mu ^{t-k}\eta_k^2\,.
\end{split}
\end{equation}
\end{lemma}
\begin{proof}
From \eqref{eq:mupdaterl}, we have
\begin{equation}
\begin{split}
\nabla f(x_t)^Tm_{t}&=\mu \nabla f(x_t)^Tm_{t-1}-\eta_t\nabla f(x_t)^Tg_{t}\,.
\end{split}
\end{equation}
Upon taking the conditional expectation with respect to $g_{t}$ conditioned on the the previous $g_1,g_2,\ldots,g_{t-1}$ and using $\mathbb{E}_t[g_t]=\nabla f(x_t)$ we arrive at
\begin{equation}\label{eq:nabfsss}
\begin{split}
\mathbb{E}_t[\nabla f(x_t)^Tm_{t}]&=\mu \nabla f(x_t)^Tm_{t-1}-\eta_t\nabla f(x_t)^T\mathbb{E}_t[g_{t}]\\
&=\mu \nabla f(x_t)^Tm_{t-1}-\eta_t\|\nabla f(x_t)\|^2\,.
\end{split}
\end{equation}
Here, the first equality follows from the fact that once we know $g_1,g_2,\ldots,g_{t-1}$, the values of $x_{t}$ and $m_{t-1}$ are not any more random.
By \eqref{eq:mupdaterl}, we obtain
\begin{equation}
\begin{split}
\|x_{t+1}-x_t\|^2&\leq\|-\lambda \eta_t g_{t}+(1-\tilde{\lambda})m_{t}\|^2\\
&\leq2\lambda^2\eta_t^2 \|g_{t}\|^2+2(1-\tilde{\lambda})^2\| m_{t}\|^2\,.
\end{split}
\end{equation}
By the $L$-Lipschitz continuity of $\nabla f$ and Cauchy-Schwarz inequality \cite{Linz}, we have
\begin{equation}
\begin{split}
[\nabla &f(x_t)-\nabla f(x_{t-1})]^Tm_{t-1}\\
&\leq\|\nabla f(x_t)-\nabla f(x_{t-1})\|\|m_{t-1}\|\\
&\leq L\|x_{t}-x_{t-1}\|\|m_{t-1}\|\leq \frac{L}{2}\|x_{t}-x_{t-1}\|^2+\frac{L}{2}\|m_{t-1}\|^2\\
&\leq L\lambda^2\eta_{t-1}^2 \|g_{t-1}\|^2+L(1-\tilde{\lambda})^2\| m_{t-1}\|^2+\frac{L}{2}\|m_{t-1}\|^2\\
&\leq L\lambda^2\eta_{t-1}^2 \|g_{t-1}\|^2+2L\|m_{t-1}\|^2\,,
\end{split}
\end{equation}
where the last inequality follows from the fact that $\tilde{\lambda}\in[0,1]$. Next,
 using the previous inequality, the estimate \eqref{eq:nabfsss} becomes
\begin{equation}
\begin{split}
&\mathbb{E}_t[\nabla f(x_t)^Tm_{t}]=\mu \nabla f(x_t)^Tm_{t-1}-\eta_t\|\nabla f(x_t)\|^2\\
&\leq\mu \nabla f(x_{t-1})^Tm_{t-1}\!-\!\eta_t\|\nabla f(x_t)\|^2+\mu [\nabla f(x_t)-\nabla f(x_{t-1})]^Tm_{t-1}\\
&\leq\mu \nabla f(x_{t-1})^Tm_{t-1}\!-\!\eta_t\|\nabla f(x_t)\|^2 +\mu  L\lambda^2\eta_{t-1}^2 \|g_{t-1}\|^2+2\mu L\|m_{t-1}\|^2.
\end{split}
\end{equation}
Upon taking the total expectation, we arrive at
\begin{equation}
\begin{split}
\mathbb{E}[\nabla f(x_t)^Tm_{t}]&\leq\mu \mathbb{E}[\nabla f(x_{t-1})^Tm_{t-1}]-\eta_t\mathbb{E}[\|\nabla f(x_t)\|^2]\\
&\quad +\mu  L\eta_{t-1}^2 \mathbb{E}[\|g_{t-1}\|^2]+2\mu L\mathbb{E}[\|m_{t-1}\|^2]\\
&\leq\mu \mathbb{E}[\nabla f(x_{t-1})^Tm_{t-1}]-\eta_t\mathbb{E}[\|\nabla f(x_t)\|^2]\\
&\quad +\mu L\lambda^2 G^2\eta_{t-1}^2+2\mu L\mathbb{E}[\|m_{t-1}\|^2]\,.
\end{split}
\end{equation}
It is now straightforward to establish by induction that
\begin{equation}
\begin{split}
\mathbb{E}\left[\nabla f(x_t)^Tm_{t}\right]&\leq -\sum_{k=1}^{t}\mu ^{t-k}\eta_k\mathbb{E}[\|\nabla f(x_k)\|^2]\\
&\quad +2L\sum_{k=1}^{t-1}\mu ^{t-k}\mathbb{E}\left\|m_{k}\right\|^2
+L\lambda^2G^2\sum_{k=1}^{t-1}\mu ^{t-k}\eta_k^2\,.
\end{split}
\end{equation}
This completes the proof.
\end{proof}

Based on the above lemmas, we obtain the main convergence theorem of the SUM method in this paper.
\begin{theorem}[Last-iterate Convergence]\label{thm:invconv}
Let $x_t$ be the sequence obtained from Algorithm 1 with a stepsize sequence satisfying \eqref{eq:etacondi}.
Assume that $f$ is lower-bounded by $f^*$, $\nabla f$ is $L$-Lipschitz continuous, $\mathbb{E}\| g_t \|^2\leq  G^2$ and $\lim_{t\to\infty}\eta_{t-1}/\eta_t=1$. Then, we have
\begin{equation}
\lim_{t\to\infty}\mathbb{E}[\|\nabla f(x_t)\|]=0,\quad \lim_{t\to\infty}\mathbb{E}[{f(x_t)}]=F^*\,,
\end{equation}
where $F^*$ is a finite constant.
\end{theorem}

\begin{proof}
By the $L$-Lipschitz continuity of $\nabla f$ and the
descent lemma in \cite{Nesterovbk}, we have
\begin{equation}
\begin{split}
f(x_{t+1})\leq f(x_t)+\nabla f(x_t)^T (x_{t+1}-x_t)+\frac{L}{2}\|x_{t+1}-x_t\|^2\,.
\end{split}
\end{equation}
Recall \eqref{eq:mupdaterl}, then the previous inequality becomes
\begin{equation}
\begin{split}
&f(x_{t+1})\leq f(x_t)+\nabla f(x_t)^T (x_{t+1}-x_t)+\frac{L}{2}\|x_{t+1}-x_t\|^2\\
&= f(x_t)-\lambda \eta_t\nabla f(x_t)^Tg_t+(1-\tilde{\lambda})\nabla f(x_t)^Tm_t+\frac{L}{2}\|x_{t+1}-x_t\|^2\,.
\end{split}
\end{equation}
Upon taking the conditional expectation with respect to $g_{t}$ conditioned on the the previous $g_1,g_2,\ldots,g_{t-1}$ and using $\mathbb{E}_t[g_t]=\nabla f(x_t)$, we arrive at
\begin{equation}\label{eq:econcfffg}
\begin{split}
\mathbb{E}_t[f(x_{t+1})]&\leq f(x_t)-\lambda\eta_t  f(x_t)^T\mathbb{E}_t[g_t]\\
&\quad+ (1-\tilde{\lambda})\mathbb{E}_t[\nabla f(x_t)^Tm_t]+\frac{L}{2}\mathbb{E}_t[\|x_{t+1}-x_t\|^2]\\
&= f(x_t)-\lambda\eta_t\|\nabla f(x_t)\|^2\\
&\quad+ (1-\tilde{\lambda})\mathbb{E}_t[\nabla f(x_t)^Tm_t]+\frac{L}{2}\mathbb{E}_t[\|x_{t+1}-x_t\|^2]\,,
\end{split}
\end{equation}
where the first equality follows from the fact that once we know $g_1,g_2,\ldots,g_{t-1}$, the value of $x_{t}$ is not any more random.
Applying the law of total expectation \cite{Linz}, we take the total expectation on both sides of \eqref{eq:econcfffg}, and arrive at
\begin{equation}\label{eq:efeftt1}
\begin{split}
&\mathbb{E}[f(x_{t+1})]\leq \mathbb{E}[f(x_t)]-\lambda\eta_t\mathbb{E}[\|\nabla f(x_t)\|^2]\\
&\quad+ (1-\tilde{\lambda})\mathbb{E}[\nabla f(x_t)^Tm_t]+\frac{L}{2}\mathbb{E}[\|x_{t+1}-x_t\|^2]\,.
\end{split}
\end{equation}
Upon inserting the result from Lemma \ref{lem:fxsinequ} into \eqref{eq:efeftt1}, we have
\begin{equation}
\begin{split}
\mathbb{E}[f(x_{t+1})]&\leq\mathbb{E}[f(x_t)]-\lambda\eta_t\mathbb{E}[\|\nabla f(x_t)\|^2]\\
&-(1-\tilde{\lambda})\sum_{k=1}^{t}\mu ^{t-k}\eta_k\mathbb{E}[\|\nabla f(x_k)\|^2]\\
&+2(1-\tilde{\lambda}) L\sum_{k=1}^{t-1}\mu ^{t-k}\mathbb{E}\left\|m_{k}\right\|^2\\
&+(1-\tilde{\lambda}) L\lambda^2G^2\sum_{k=1}^{t-1}\mu ^{t-k}\eta_k^2+\frac{L}{2}\mathbb{E}[\|x_{t+1}-x_t\|^2]\,.
\end{split}
\end{equation}
From Lemma \ref{lem:smbound},  we know that
\begin{equation}
\sum_{t=1}^\infty\mathbb{E}\|x_{t+1}-x_t\|^2<\infty
,\quad
\sum_{t=1}^\infty\left(\sum_{k=1}^{t-1}\mu ^{t-k}\eta_k^2\right)<\infty\,,
\end{equation}
and
\begin{equation}
\sum_{t=1}^\infty \left(\sum_{k=1}^{t-1}\mu ^{t-k}\mathbb{E}\|m_{k}\|^2\right)<\infty\,.
\end{equation}
Thus, the conditions of Lemma \ref{lem:xyz} are satisfied for $Y_t=\mathbb{E}[f(x_t)]$, $W_t=\lambda\eta_t\mathbb{E}[\|\nabla f(x_t)\|^2]{+}(1-\tilde{\lambda})\sum_{k=1}^{t}\mu ^{t-k}\eta_k\mathbb{E}[\|\nabla f(x_k)\|^2]$  and $Z_t=$ $\frac{L}{2}\mathbb{E}[\|x_{t+1}-x_t\|^2]$+$2(1-\tilde{\lambda}) L\sum_{k=1}^{t-1}\mu ^{t-k}\mathbb{E}\left\|m_{k}\right\|^2$+$(1-\tilde{\lambda}) L\lambda^2G^2\sum_{k=1}^{t-1}\mu ^{t-k}\eta_k^2$. By Lemma \ref{lem:xyz}, we have
\begin{equation}
\lim_{t\to\infty}Y_t=\lim_{t\to\infty}\mathbb{E}[{f(x_t)}]=F^*\,,
\end{equation}
where $F^*$ is a finite constant and
\begin{equation}\label{eq:etakektozero}
\begin{split}
\sum_{t=1}^\infty\Bigl(\lambda\eta_t\mathbb{E}[\|\nabla f(x_t)\|^2]+(1-\tilde{\lambda})\sum_{k=1}^{t}\mu ^{t-k}\eta_k\mathbb{E}[\|\nabla f(x_k)\|^2]\Bigr)<\infty\,.
\end{split}
\end{equation}
Since
\begin{equation}\label{eq:beltaphefx}
\begin{split}
\sum_{t=1}^T\sum_{k=1}^{t}\mu ^{t-k}\eta_k\mathbb{E}[\|\nabla f(x_k)\|^2]&=\sum_{k=1}^T\eta_k\mathbb{E}[\|\nabla f(x_k)\|^2]\sum_{t=k}^{T}\mu ^{t-k}\\
&\geq \sum_{k=1}^T\eta_k\mathbb{E}[\|\nabla f(x_k)\|^2]\,.
\end{split}
\end{equation}
where the last inequality follows from the fact that $\sum_{t=k}^{T}\mu ^{t-k}\geq 1$. We can now combine \eqref{eq:etakektozero} and \eqref{eq:beltaphefx}, to yield
\begin{equation}
\sum_{t=1}^\infty\left((1-\tilde{\lambda}+\lambda)\eta_t\mathbb{E}[\|\nabla f(x_t)\|^2]\right)<\infty\,.
\end{equation}
Note that $\tilde{\lambda}=(1-\mu)\lambda$, then $1-\tilde{\lambda}+\lambda=1+\mu\lambda\geq1$ is a non-zero constant. From Jensen's inequality \cite{Linz}, it follows that
\begin{equation}\label{eq:alphnabfxk}
\sum_{t=1}^\infty\eta_t\big(\mathbb{E}[\|\nabla f(x_t)\|]\big)^2\leq\sum_{t=1}^\infty\eta_t\mathbb{E}[\|\nabla f(x_t)\|^2]<\infty\,,
\end{equation}
while by involving \eqref{eq:mupdaterl} and \eqref{eq:mtisumbag}, we have
\begin{equation}
\begin{split}
\|x_{t+1}-x_t\|&=\|-\lambda\eta_t g_{t}+(1-\tilde{\lambda}) m_{t}\|\leq\lambda\eta_t\|g_{t}\|+(1-\tilde{\lambda})\| m_{t}\|\\
&\leq\lambda\eta_t\|g_{t}\|+(1-\tilde{\lambda})\sum_{k=1}^t\mu ^{t-k}\eta_k\|g_{k}\|\,.
\end{split}
\end{equation}
Upon taking the total expectation, noting that $\lambda\in[0,1/(1-\mu)]$, we finally obtain
\begin{equation}\label{eq:mtbtkakgk}
\begin{split}
&\mathbb{E}[\|x_{t+1}-x_t\|]\leq\lambda\eta_t\mathbb{E}[\|g_{t}\|]+(1-\tilde{\lambda})\sum_{k=1}^t\mu ^{t-k}\eta_k\mathbb{E}[\|g_{k}\|]\\
&\leq \frac{2G}{1-\mu}\left(\frac{\eta_t}{2}+\frac{(1-\mu)\sum_{k=1}^t\mu ^{t-k}\eta_k}{2}\right)=\frac{2G}{(1-\mu)}\frac{\eta_t+\tilde{\eta}_t}{2}\,,
\end{split}
\end{equation}
where $\tilde{\eta}_t=(1-\mu)\sum_{k=1}^t\mu ^{t-k}\eta_k$. By the Stolz theorem of real sequence \cite{Rudin}, we next have
\begin{equation}
\begin{split}
\lim_{t\to\infty}\frac{\eta_t}{\tilde{\eta}_t}&=\frac{1}{1-\mu}
\lim_{t\to\infty}\frac{\eta_t/\mu^t}{\eta_t/\mu^t+\eta_{t-1}/\mu^{t-1}+\cdots+\eta_1/\mu}\\
&=\frac{1}{1-\mu}\lim_{t\to\infty}\frac{\eta_t/\mu^t-\eta_{t-1}/\mu^{t-1}}{\eta_t/\mu^t}\\
&=\frac{1}{1-\mu}\lim_{t\to\infty}\left(1-\mu\frac{\eta_{t-1}}{\eta_{t}}\right)=1\,,
\end{split}
\end{equation}
where the last equality follows from the condition $\lim_{t\to\infty}\eta_{t-1}/\eta_t=1$. By the $L$-Lipschitz continuity of $\nabla f$, we have
\begin{equation}
\begin{split}
\big|\|\nabla f(x_{t+1})\|-\|\nabla f(x_t)\|\big|&\leq\|\nabla f(x_{t+1})-\nabla f(x_t)\|\\
&\leq L\| x_{t+1}-x_t\|\leq L\|x_{t+1}-x_t\|\,.
\end{split}
\end{equation}
Upon taking the total expectation and using the Jensen's inequality \cite{Linz}, we arrive at
\begin{equation}\label{eq:ffalpsamm}
\begin{split}
\Big|\mathbb{E}[\|\nabla f(x_{t+1})\|]-\mathbb{E}[\|\nabla f(x_t)\|]\Big|&\leq
\mathbb{E}\left[\Bigl|\|\nabla f(x_{t+1})\|-\|\nabla f(x_t)\|\Bigr|\right]\\
&\leq L\mathbb{E}[\| x_{t+1}-x_t\|]\\
&\leq\frac{2LG}{(1-\mu)}\frac{\eta_t+\tilde{\eta}_t}{2}\,.
\end{split}
\end{equation}
It follows from \eqref{eq:alphnabfxk} and \eqref{eq:ffalpsamm} that the conditions of Corollary \ref{cor:abcabp} are satisfied for $a_t=\eta_t$,
$\tilde{a}_t=(\eta_t+\tilde{\eta}_t)/{2}$ and $b_t=\mathbb{E}[\|\nabla f(x_t)\|]$. Therefore,
\begin{equation}
\lim_{t\to\infty}\mathbb{E}[\|\nabla f(x_t)\|]=0\,.
\end{equation}
This completes the proof.
\end{proof}

\begin{remark}
It should be noted that, including but not
limited to, there is a large class of non-increasing sequences
that satisfy the step size conditions of Theorem \ref{thm:invconv}, i.e.,
\begin{equation}\label{eq:noncrsstepsize}
\eta_t=
    \begin{cases}
      \alpha,\quad &t\leq K \\
      \alpha/(t-K)^p,\quad &t> K
    \end{cases},\quad  p\in(1/2,1]\,.
\end{equation}
where $\alpha$ and $K$ are two positive constants.
\end{remark}

The following result falls as an immediate corollary of Theorem \ref{thm:invconv}.
\begin{corollary}
Suppose that the conditions in Theorem \ref{thm:invconv} hold. Then,
\begin{equation}
\lim_{T\to\infty}\min_{t\in[T]}\mathbb{E}\|\nabla f(x_t)\|^2=0,\quad \lim_{T\to\infty}\frac{1}{T}\sum_{t=1}^T\mathbb{E}\|\nabla f(x_t)\|^2=0\,.
\end{equation}
\end{corollary}

\section{Experiments}
This section validates the convergence theory of the proposed SUM method for training
neural networks, and for all the tested algorithms,
we take the sequence of step sizes according to \eqref{eq:noncrsstepsize}.
By convention, the last-iterate convergences of the SUM method are evaluated via the training loss and test accuracy versus the number of epochs, respectively. For a fair comparison, every algorithm used the same amount of training data, the same initial step size $\alpha$, and the same initial weights.
We run all the experiments on a server with AMD Ryzen TR 2990WX CPU, 64GB RAM, and one NVIDIA GTX-2080TI GPU using a public accessible code\footnote{https://github.com/kuangliu/pytorch-cifar/}
with Python 3.7.9 and Pytorch 1.8.1.

\subsection{Dataset}

\noindent {\bf MNIST Dataset:} The MNIST dataset \cite{Lenet}
contains 60,000 training samples and 10,000 test samples of the handwritten digits from 0 to 9.
The images are grayscale, $28\times28$ pixels, and centered to reduce preprocessing.

\noindent {\bf CIFAR-10 Dataset:} The CIFAR-10 dataset \cite{CIFAR} consists of 60,000 $32\times32$ color images drawn
from 10 classes, which is divided into a traing set with 50,000 images and a test set with 10,000 images.

\subsection{Results of LeNet on MNIST dataset}

\begin{figure}[htbp]
  \centering
  \subfigure[Training loss]{\includegraphics[width=2.6in]{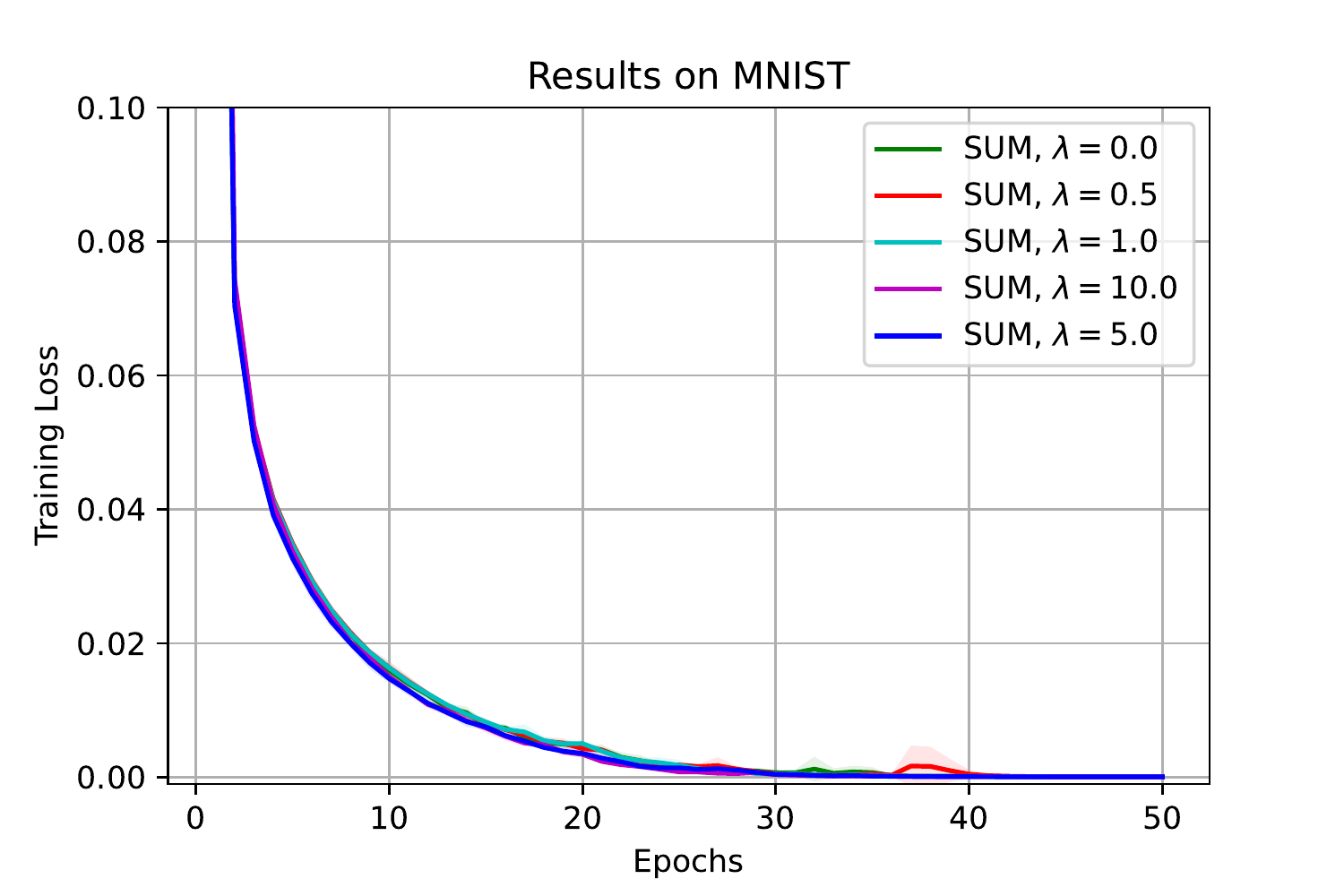}}
  \hspace{-0.2in}
  \subfigure[Test accuracy]{\includegraphics[width=2.6in]{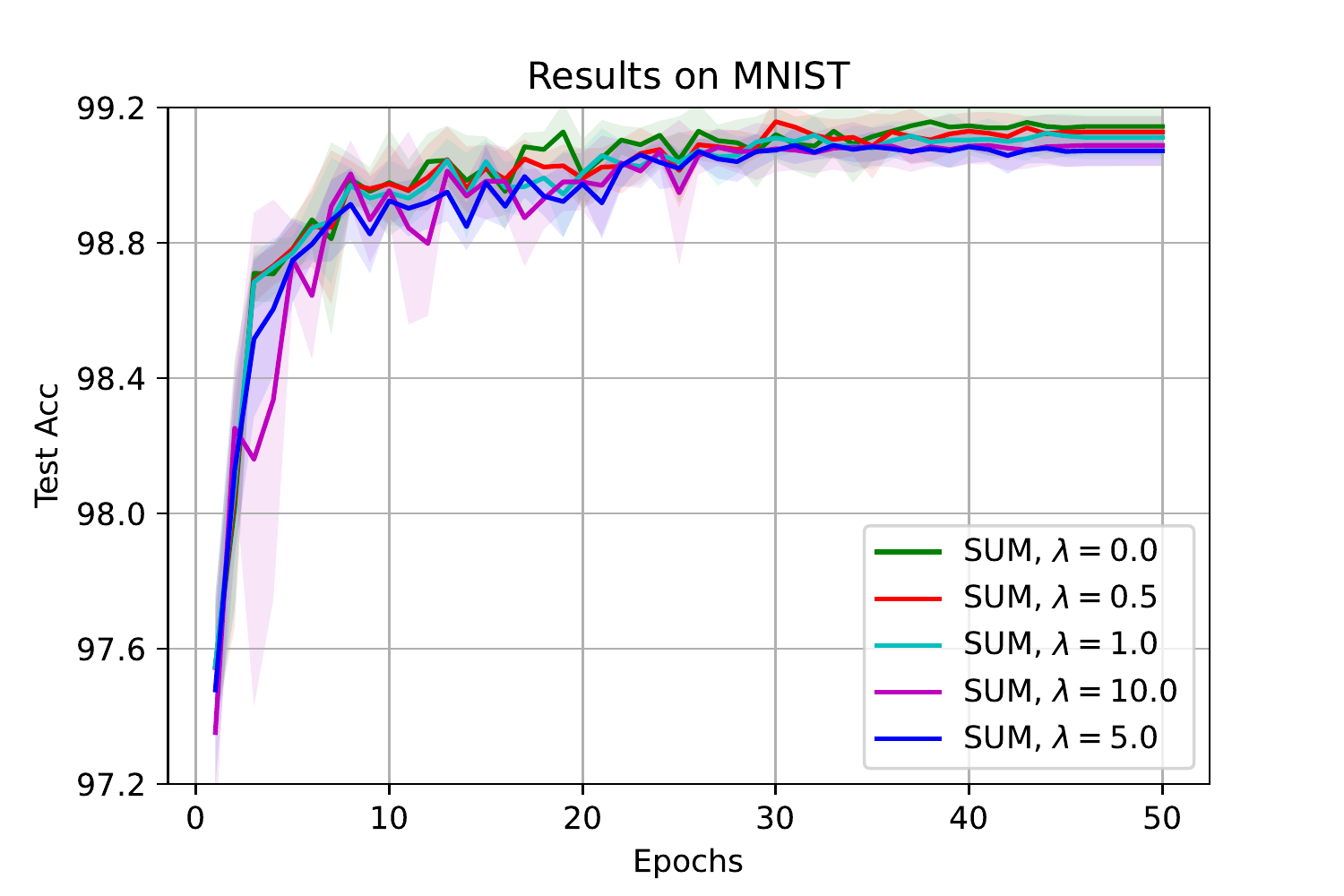}}
    \hspace{-0.2in}
  \caption{Results of LeNet-5 on MNIST.
  The continuous lines and shadow areas represent the mean and standard deviation from 5 random initializations, respectively.
  }
  \label{fig:mnist}
\end{figure}

We first trained LeNet \cite{Lenet} on the MNIST dataset using SUM method with the initial step size $\alpha=0.01$ and the momentum constant $\mu=0.9$.
In particular, we compared SUM with the interpolation factor $\lambda\in\{0, 0.5, 1.0,$ $5.0, 10.0\}$, where $\lambda=0$ corresponds to SHB, $\lambda=1$ corresponds to SNAG and
$\lambda=10$ is the maximum value of $1/(1-\mu)$ because the momentum constant was $\mu = 0.9$.
LeNet constants of 2 convolutional (Conv) and max-pooling layers and 3 fully connected (FC) layers, and the activation functions
of the hidden layers are taken as ELU functions \cite{Calinbook}. The training was on mini-batches with
128 images per batch for 50 epochs, a total of $T=23,450$ iterations\footnote{The number of iterations is equal to the number of training samples divided by the batch size, and then multiplied by the number of epochs.}, and $K$ is set to $0.9T$.
The performance of the SUM method with different interpolation factors is shown in Figure \ref{fig:mnist}.
We can see that all the variants of the SUM method had similar performance in terms of the training loss, which is consistent with the result in Theorem \ref{thm:invconv}. For the test part, we found that SUM (minimum $\lambda=0.0$) yields a higher test accuracy than other variants, while SUM (maximum $\lambda=10.0$) exhibits larger oscillations.
An interesting phenomenon is that non-increasing step size (recall \eqref{eq:noncrsstepsize}) has the effect of
smoothing on the tail of the test accuracies.

\subsection{Results on CIFAR-10 dataset}
\begin{figure}[htbp]
  \centering
  \subfigure[Training loss]{\includegraphics[width=2.6in]{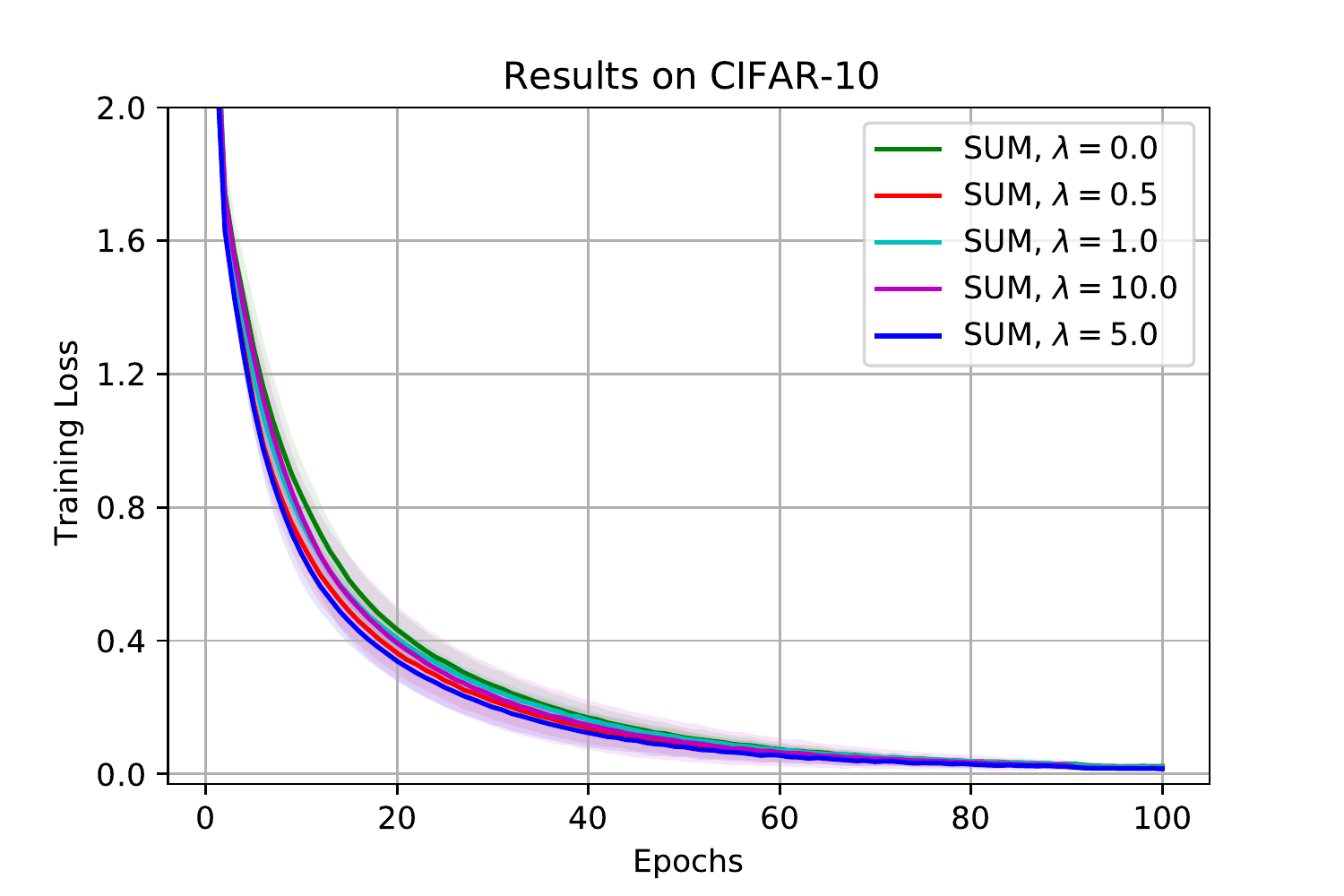}}
  \hspace{-0.2in}
  \subfigure[Test accuracy]{\includegraphics[width=2.6in]{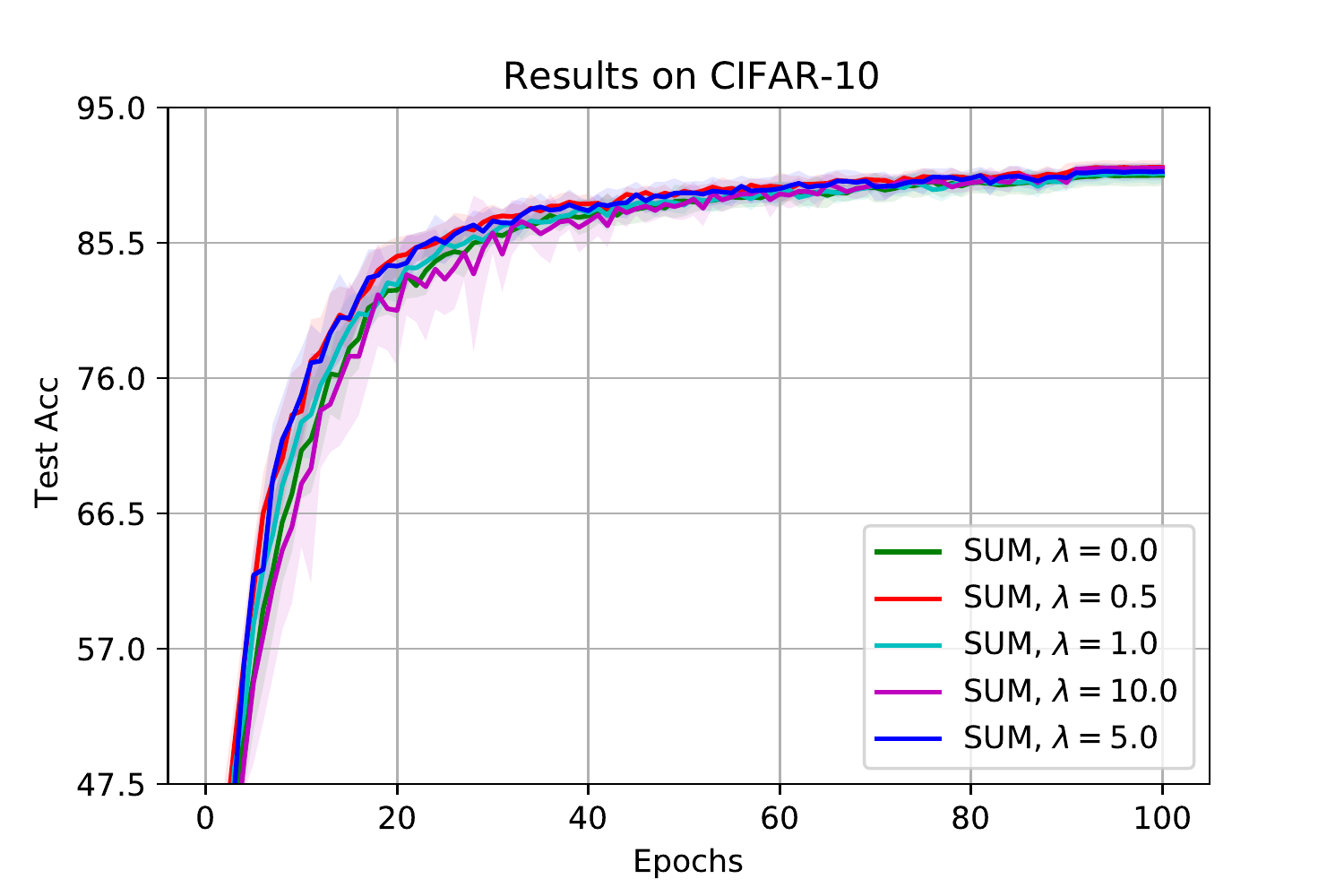}}
  \caption{Results of ResNet-18 on CIFAR-10.
  The continuous lines and shadow areas represent the mean and standard deviation from 5 random initializations, respectively.}
  \label{fig:cifar}
\end{figure}

We next verify the convergence of the SUM method
on CIFAR-10 by using ResNet-18 \cite{Hekm}.
ResNet-18 contains a Conv layer, 4 convolution blocks with [2,2,2,2] layers, one FC layer, and the activation functions
of the hidden layers are taken as ELU functions \cite{Calinbook}.
The batch size is 128. The training stage lasts for 100 epochs, a total of $T=39,100$ iterations, and $K=0.9T$.
We fix the momentum constant $\mu=0.9$ and the initial step size $\alpha=0.1$.
Now, we study the influence of $\lambda$ on convergence of the SUM method by selecting $\lambda$ from the collection set $\{0, 0.5, 1.0, 5.0, 10.0\}$.
Figure \ref{fig:cifar} illustrates the performance of the SUM with different $\lambda$ on the CIFAR-10 dataset. From
Figure \ref{fig:cifar}(a), we can observe that the convergence of the SUM method can be guaranteed on the training set, as long as the interpolation factor $\lambda\in[0,1/(1-\mu)]$, which coincides with the convergence results in Theorem \ref{thm:invconv}. Figure \ref{fig:cifar}(b) shows that the test accuracies
of SUM ($\lambda=0.5$ and $\lambda=5.0$) are comparable, followed by SUM ($\lambda=0.0$ and $\lambda=1.0$), and that SUM ($\lambda=10.0$) exhibits larger oscillations similar to the MNIST dataset. It should be noted that for theoretical consistency, the training loss was calculated after each epoch on the training set (not mini-batches), and so it was reasonable that the loss curves of the training set were monotonously decreasing.


\section{Conclusions}
We have addressed the last-iterate convergence of the stochastic momentum methods in a unified
framework, covering both SHB and SNAG. For rigour, this has been archived: \textit{i)} in a non-convex optimization setting, \textit{ii)} under the condition that the momentum coefficient is constant, and \textit{iii)} without the assumption about the boundedness of the weight norm.
Moreover, we have established that the existing minimum convergence and random selection convergence are both corollaries of last-iterate convergence, thus providing a new perspective on the understanding of the convergence of SUM method. We have also experimentally
tested the convergence of SUM on the benchmark datasets.
Experimental results verify the theoretical analysis.





\bibliographystyle{model1a-num-names}
\bibliography{<your-bib-database>}

\begin{thebibliography}{00}

\bibitem{LeCun}
Y. LeCun, Y. Bengio, G. Hinton,
Deep learning, Nature 521 (2015) 436-444.

\bibitem{Schmidhuber}
J. Schmidhuber,
Deep learning in neural networks: An overview,
Neural Netw. 61 (2015) 85-117.


 \bibitem{Mandicbk}D. Mandic, J. Chambers,
Recurrent Neural Networks for Prediction: Architectures, Learning Algorithms and Stability, Wiley, New York, NY, USA, 2001.

\bibitem{Robbins} H. Robbins, S. Monro,
A stochastic approximation method,
Ann. Math. Statistics 22 (3) (1951) 400-407.

\bibitem{Luo22}
J. Luo, J. Liu, D. Xu, H. Zhang,
SGD-r$\alpha$: A real-time $\alpha$-suffix averaging method for SGD with biased gradient estimates,
Neurocomputing 487 (2022) 1-8.



\bibitem{Bottou07} L. Bottou, O. Bousquet,
The tradeoffs of large scale learning,
in: Proc. Adv. Neural Inf. Process. Syst. 2007, pp. 161-168.



\bibitem{Duchi}
J. Duchi, E. Hazan, Y. Singer,
Adaptive subgradient methods for online learning and stochastic optimization,
J. Mach. Learn. Res. 12 (2011) 2121-2159.

\bibitem{Tieleman}
T. Tieleman, G. Hinton,
Lecture 6.5-RMSProp: Divide the gradient by a running average of its recent magnitude.
COURSERA: Neural Netw. Mach. Learn. 4 (2012) 26-31.

\bibitem{Xu21}
D. Xu, S. Zhang, H. Zhang, D. Mandic,
Convergence of the RMSProp deep learning method with penalty for nonconvex optimization,
Neural Netw. 139 (2021) 17-23.

\bibitem{Nesterov} Y. Nesterov,
A method for solving the convex programming problem with convergence rate O($1/k^2$),
in: Dokl. Akad. Nauk SSSR 1983, pp. 543-547.

\bibitem{Polyak} B. Polyak,
Some methods of speeding up the convergence of iteration methods,
USSR Comp. Math. Math. Phys. 4 (5) (1964) 1-17.



\bibitem{GhadimiE}
E. Ghadimi, H. Feyzmahdavian, M. Johansson,
Global convergence of the heavy-ball method for convex optimization,
in: Europ. Cont. Conf. 2015, pp. 310-315.

\bibitem{Kingma}
D. Kingma, J. Ba,
Adam: A method for stochastic optimization,
in: Int. Conf. Learn. Repres. 2015, pp. 1-15.





\bibitem{Luo}
L. Luo, Y. Xiong, Y. Liu, Y. Sun,
Adaptive gradient methods with dynamic bound of learning rate,
in: Int. Conf. Learn. Repres. 2019, pp. 1-21.

\bibitem{Xu22}
J. Liu, J. Kong, D. Xu, M. Qi, Y. Lu,
Convergence analysis of AdaBound with relaxed bound functions for non-convex optimization,
Neural Netw. 145 (2022) 300-307.



\bibitem{Ghadimi} S. Ghadimi, G. Lan,
Accelerated gradient methods for nonconvex nonlinear and stochastic programming,
Math. Program. 156 (1-2) (2016) 59-99.


\bibitem{FZou} F. Zou, L. Shen, Z. Jie, J. Sun, W. Liu,
Weighted AdaGrad with unified momentum, ArXiv preprint arXiv: 1808.03408, 2018.


\bibitem{YYan} Y. Yan, T. Yang, Z. Li, Q. Lin, Y. Yang,
A unified analysis of stochastic momentum methods for deep learning,
in: Proc. Int. Joint Conf. Artif. Intell. 2018, pp. 2955-2961.

\bibitem{ChenX}
X. Chen, S. Liu, R. Sun, M. Hong,
On the convergence of a class of Adam-type algorithms for non-convex optimization,
in: Int. Conf. Learn. Repres. 2019, pp. 1-30.


\bibitem{Zou} F. Zou,  L. Shen,  Z. Jie,  W. Zhang,  W. Liu,
A sufficient condition for convergences of Adam and RMSProp,
in: Proc. IEEE Conf. Comp. Vis. Patt. Recogn. 2019, pp. 11127-11135.

\bibitem{Wangj} J. Wang, J. Yang, W. Wu,
Convergence of cyclic and almost-cyclic learning with momentum for feedforward neural networks,
IEEE Trans. Neural Netw. 22 (8) (2011) 1297-1306.

\bibitem{Zhangnm} N. Zhang,
An online gradient method with momentum for two-layer feedforward neural networks,
Appl. Math. Comput. 212 (2) (2009) 488-498.

\bibitem{Stevens}
E. Stevens, L. Antiga, T. Viehmann,
Deep Learning with PyTorch, Manning Publications, 2020.

\bibitem{Zaccone}
G. Zaccone, M. R. Karim, A. Menshawy,
Deep Learning with TensorFlow, Packt Publishing Ltd, 2017.


\bibitem{Bertsekas} D. Bertsekas, J. Tsitsiklis,
Gradient convergence in gradient methods with errors,
SIAM J. Optim. 10 (3) (2000) 627-642.

\bibitem{Linz}
Z. Lin, Z. Bai,
Probability Inequalities, Springer Science and Business Media, 2011.

\bibitem{Rudin}
W. Rudin,
Principles of Mathematical Analysis, 2nd ed. New York: McGraw-Hill, 1964.

\bibitem{Nesterovbk} Y. Nesterov,
Introductory Lectures on Convex Optimization: A Basic Course, Springer Science and Business Media, 2013.

\bibitem{Lenet}
Y. LeCun, L. Bottou, Y. Bengio, P. Haffner,
Gradient-based learning applied to document recognition,
Proc. IEEE. 86 (11) (1998) 2278-2324.


\bibitem{CIFAR}
A. Krizhevsky,
Learning multiple layers of features from tiny images, Master's thesis, University of Toronto, 2009.

\bibitem{Calinbook}
O. Calin,
Deep Learning Architectures: A Mathematical Approach, Springer, 2020.

\bibitem{Hekm} K. He, X. Zhang, S. Ren, J. Sun,
Deep residual learning for image recognition,
in: Proc. IEEE Conf. Comp. Vis. Patt. Recogn. 2016, pp. 770-778.






\bibitem{Goodfellow}I. Goodfellow, Y. Bengio, A. Courville,
Deep Learning,  Cambridge, MA: MIT Press, 2016.













































\end{thebibliography}



\end{document}